\documentclass[11pt,a4]{article}

\usepackage{amsthm,amssymb}
\usepackage{amsmath,yhmath}
\usepackage{mathrsfs}
\usepackage{hyperref,graphicx,color}
\usepackage[margin=3cm]{geometry}

\hypersetup{linkcolor=black,citecolor=black,colorlinks=true}

\theoremstyle{plain}
\newtheorem{theorem}{Theorem}[section]
\newtheorem{corollary}[theorem]{Corollary}
\newtheorem{lemma}[theorem]{Lemma}

\theoremstyle{definition}

\theoremstyle{remark}

\begin{document}

\title{Notes of Boundedness on Cauchy Integrals on Lipschitz Curves
  ($p=2$)}
\author{Deng Guantie\thanks{
      E-mail: denggt@bnu.edu.cn, School of Mathematical Sciences, 
      Beijing Normal University, Beijing, China.} 
    \ and Liu Rong\thanks{
      Corresponding author, E-mail: rong.liu@mail.bnu.edu.cn
      School of Mathematical Sciences, 
      Beijing Normal University, Beijing, China.}
}

\maketitle

\begin{abstract}
  We provide the details of the first proof in~\cite{CJS89}, which proved 
  that Cauchy transform of $L^2$~functions on Lipschitz curves is bounded.
  We then prove that every $L^2$~function on Lipschitz curves is the sum
  of non-tangential boundary limit of functions in $H^2(\Omega_\pm)$, 
  the Hardy spaces on domains over and under the Lipschitz curve.
  We also obtain a more accurate boundary of Cauchy transform under the
  condition that the Lipschitz curve is the real axis. 
\end{abstract}

{\bf Keywords:}
  Cauchy Integral, Lipschitz curve, Hardy space, 
  non-tangential boundary limit,

{\bf 2010 Mathematics Subject Classification:}
  Primary: 30H10, Secondary: 30E20, 30E25 

\section{Introduction}
  Paper~\cite{CJS89} offered two elementary proofs of the boundedness of 
  Cauchy integral (or transform) on Lipschitz curve~$\Gamma$, with integral 
  index~$p=2$. The first one in which we are interested is succinct, 
  thus without many details. 
  In this paper, we give the full version of that proof. Since the Cauchy 
  integral is actually analytic on two domains over and under $\Gamma$ 
  which we denote as $\Omega_\pm$, it is in $H^2(\Omega_\pm)$, the 
  Hardy spaces on $\Omega_\pm$, hence has non-tangential
  boundary limits from above and below $\Gamma$~\cite{DL171}. Then we could 
  reach the result that every function in $L^2(\Gamma)$ is the sum of two 
  functions in $H^2(\Omega_+)$ and $H^2(\Omega_-)$, respectively. That result
  is usually written as $L^2(\Gamma)= H^2(\Omega_+)+H^2(\Omega_-)$. 
  We also apply
  the same method to the special case of $\Gamma$ be $\mathbb{R}$, and obtain
  a more accurate boundary of the Cauchy transform.
\section{Definitions}

Let $\Gamma=\{\zeta(u)=u+\mathrm{i}a(u)\colon u\in\mathbb{R}\}$ be 
a Lipschitz curve in the complex plain~$\mathbb{C}$, where 
$\lVert a'\rVert_\infty = M< \infty$, 
$\Omega_{\pm}= \{\zeta(u)\pm\mathrm{i}\tau\colon u\in\mathbb{R}, \tau>0\}$ 
be the two domains lying above and below $\Gamma$, 
$\Phi_{\pm}$ be the two conformal representations from 
$\mathbb{C}_{\pm}=\{x+\mathrm{i}y\colon x\in\mathbb{R}, y>0\}$ 
onto $\Omega_{\pm}$, which both extend to the boundary, such that
$\Phi_{\pm}(\mathbb{R})= \Gamma$ and $\Phi_{\pm}(\infty)= \infty$, 
$\Psi_{\pm}\colon \Omega_{\pm}\to\mathbb{C}_{\pm}$ be the inverse mappings 
of $\Phi_{\pm}$. More detail of $\Phi_\pm$ and $\Psi_\pm$ are in~\cite{Ke80}
and \cite{DL171}.

For $w\in\mathbb{C}$, define $d(w)$ as the distance from $w$ to 
the curve $\Gamma$, that is 
\[d(w)= \inf\{|w-\zeta|\colon \zeta\in\Gamma\},\]
which implies that  
\[d(w)\leqslant |w-\zeta(0)|\leqslant |w|+|a(0)|.\]

For $r>0$, denote $\{|z|<r\colon z\in\mathbb{C}\}$ 
as $D(0,r)$, and $D(0,1)$, the unit disk of $\mathbb{C}$, as $\mathbb{D}$.
For domain $D\subset\mathbb{C}$ and measure $\mathrm{d}m$ on it, 
let $L^2(D,\mathrm{d}m)$ be the function space of all complex valued, 
$\mathrm{d}m$~measurable functions on $D$, equipped with norm
\[\lVert f\rVert_{L^2(D,\mathrm{d}m)}
  = \Big(\int_D |f|^2\,\mathrm{d}m\Big)^{\frac12},\quad
  \text{for } f\in L^2(D,\mathrm{d}m).\] 
Thus we could consider function spaces like
$L^2(\mathbb{R},\mathrm{d}x)$, $L^2(\mathbb{C}_\pm,\mathrm{d}\lambda)$, 
where $\mathrm{d}\lambda$ is the area measure on $\mathbb{C}$, 
and a few more which will appear later in this paper.

Let $F(w)$ be a function analytic on $\Omega_+$, if
\[\sup_{\tau>0}\Big(\int_{\Gamma_\tau} |F(w)|^p |\mathrm{d}w|\Big)^{\frac1p}
  = \lVert F\rVert_{H^p(\Omega_+)}<\infty,\quad
  \text{for } 0<p<\infty,\]
where $\Gamma_\tau= \{\zeta+\mathrm{i}\tau\colon \zeta\in\Gamma\}$, 
then we say that $F(w)\in H^p(\Omega_+)$.
  
Fix $u_0\in\mathbb{R}$ such that 
$\zeta'(u_0)=|\zeta'(u_0)|\mathrm{e}^{\mathrm{i}\phi_0}$ exists, and 
choose $\phi\in(0,\frac{\pi}2)$, we denote $\zeta_0=\zeta(u_0)$ and let
\[\Omega_{\phi}(\zeta_0)
  =\{\zeta_0+r\mathrm{e}^{\mathrm{i}\theta}
     \colon r>0, \theta-\phi_0\in(\phi,\pi-\phi)\},\]
then we say that a function $F(w)$ on $\Omega^+$ has non-tangential boundary 
limit~$l$ at $\zeta_0$ if for $w\in\Omega_{\phi}(\zeta_0)\cap\Omega^+$,
\[\lim_{w\to\zeta_0}
  F(w)= l, \quad \text{for any } \phi\in\Big(0,\frac{\pi}2\Big).\]
It is not difficult to verify that for fixed $\phi\in(0,\frac{\pi}2)$, 
there exists constant $\delta>0$, such that, if $|z|<\delta$ and 
$\zeta_0+z\in\Omega_\phi(\zeta_0)$, then $\zeta_0+z\in\Omega^+$ and 
$\zeta_0-z\in\Omega^-$.

Let $\zeta$, $\zeta_0\in\Gamma$, we define
\[K_z(\zeta,\zeta_0)
  = \frac1{2\pi\mathrm{i}}\bigg(\frac1{\zeta-(\zeta_0+z)}
       - \frac1{\zeta-(\zeta_0-z)}\bigg),\]
for $z\in\mathbb{C}$ and $z\neq \pm(\zeta-\zeta_0)$, 
then $K_z(\zeta,\zeta_0)$ is well-defined and we could write 
\begin{equation}\label{equ-170623-1150}
  K_z(\zeta,\zeta_0)
  = \frac1{\pi\mathrm{i}}\cdot\frac{z}{(\zeta-\zeta_0)^2-z^2},
\end{equation}
We could also verify that, if $\zeta_0+z\in\Omega^+$ and 
$\zeta_0-z\in\Omega^-$, then 
\[\int_{\Gamma} K_z(\zeta,\zeta_0)\,\mathrm{d}\zeta= 1.\]

\section{Lemmas}

Let $f$ be a univalent holomorphic function on $\mathbb{D}$. 
If $f(0)=0$ and $f'(0)=1$, we say that $f$ is in $\mathcal{S}$. 
In other words, we define 
\[\mathcal{S}= \{f(z)= z+a_2 z^2+\cdots\colon
    f \text{ is holomorphic and univalent on } \mathbb{D}\}.\]

\begin{lemma}\label{lem-170726-1025}
  If $f\in\mathcal{S}$ and is continuous to the boundary $\partial\mathbb{D}$, 
  then 
  \[\frac14\leqslant \inf\{|f(z)|\colon |z|=1\}\leqslant 1.\]
\end{lemma}

\begin{proof}
  The first part of the above inequality comes from the Koebe $\frac14$ 
  theorem~\cite{De10}, and we only need to prove the second part. Let 
  $\inf\{|f(z)|\colon |z|=1\}= b$, then $b\geqslant\frac14$ and 
  $\overline{D(0,b)}\subset f(\overline{\mathbb{D}})$, 
  or $f^{-1}(\overline{D(0,b)})\subset \overline{\mathbb{D}}$. 
  Define $g(z)= f^{-1}(bz)$ on $\mathbb{D}$, then 
  $g(\mathbb{D})\subset \overline{\mathbb{D}}$, and $g(0)=f^{-1}(0)=0$. 
  By Schwarz lemma, $|g'(0)|\leqslant 1$. Since $f'(0)=1$, 
  $g'(z)= (f^{-1})'(bz)\cdot b$, and $(f^{-1})'(0)=(f'(0))^{-1}=1$,
  we have $|g'(0)|= |b|\leqslant 1$, and the lemma is proved.
\end{proof}

Some results below contain ``$\pm$'' as subscript and the two cases usually 
could be proved by using the same method. Then we will prove only one case 
and write the other one as a corollary.

\begin{lemma}\label{lem-170727-1550}
  If $z=x+\mathrm{i}y\in\mathbb{C}_+$, where $x\in\mathbb{R}$, $y>0$, 
  then
  \[|y\Phi_+'(z)|
    \leqslant 2d(\Phi_+(z))
    \leqslant 4|y\Phi_+'(z)|,\]
  and
  \[|y\Phi_+''(z)|\leqslant 3|\Phi_+'(z)|.\]
  Consequently,
  \[|y^2\Phi_+''(z)|\leqslant 6d(\Phi_+(z)).\]
\end{lemma}

\begin{proof}
  Fix $z_0=x_0+\mathrm{i}y_0\in\mathbb{C}_+$, and define
  \[z= T(\xi)= \frac{\overline{z_0}\xi-z_0}{\xi-1},\quad
    \text{for } \xi\in\mathbb{D},\]
  then $T$ is a fractional linear mapping from $\mathbb{D}$ onto 
  $\mathbb{C}_+$, $T(0)=z_0$, and
  \[T'(\xi)= \frac{z_0-\overline{z_0}}{(\xi-1)^2},\quad 
    \xi= T^{-1}(z)= \frac{z-z_0}{z-\overline{z_0}}.\]
  
  Denote $\Phi_+$ as $\Phi$ and let
  \begin{equation}\label{equ-170727-1020}
    f(\xi)= \frac{\Phi(T(\xi))-\Phi(z_0)}{(z_0-\overline{z_0})\Phi'(z_0)}
  \end{equation}
  for $\xi\in\mathbb{D}$, then $f$ is univalent, $f(0)=0$, and 
  \[f'(\xi)
    = \frac{\Phi'(T(\xi))\cdot T'(\xi)}{(z_0-\overline{z_0})\Phi'(z_0)}
    = \frac{\Phi'(T(\xi))}{(\xi-1)^2 \Phi'(z_0)},\]
  thus $f'(0)=1$ and $f\in\mathcal{S}$. By Lemma~\ref{lem-170726-1025}, 
  \[\frac14\leqslant \inf\{|f(\xi)|\colon |\xi|=1\}\leqslant 1,\]
  which is, by~\eqref{equ-170727-1020} and 
  $z_0-\overline{z_0}= 2\mathrm{i}y_0$,
  \[\frac12|y_0\Phi'(z_0)|
    \leqslant \inf\{|\Phi(z)-\Phi(z_0)|\colon z\in\mathbb{R}\}
    \leqslant 2|y_0\Phi'(z_0)|,\]
  and it follows that
  \[|y_0\Phi'(z_0)|\leqslant 2d(\Phi(z_0))\leqslant 4|y_0\Phi'(z_0)|.\]

  We also have, by Bieberbach theorem, 
  \begin{equation}\label{equ-170727-1040}
    \Bigl|\frac{f''(0)}{2!}\Bigr|\leqslant 2,\quad
    \text{or } |f''(0)|\leqslant 4.
  \end{equation}
  Since
  \[f''(\xi)= \frac1{(\xi-1)^3 \Phi'(z_0)}
      \big(\Phi''(T(\xi))\cdot T'(\xi)(\xi-1)- 2\Phi'(T(\xi))\big),\]
  $T(0)=z_0$ and $T'(0)= z_0-\overline{z_0}$, then
  \begin{align*}
    f''(0)
    &= \frac1{\Phi'(z_0)} \big(\Phi''(z_0)(z_0-\overline{z_0})
         + 2\Phi'(z_0)\big)                                     \\
    &= 2\mathrm{i}y_0 \frac{\Phi''(z_0)}{\Phi'(z_0)} +2,
  \end{align*}
  and, by \eqref{equ-170727-1040},
  \[\biggl|2\mathrm{i}y_0 \frac{\Phi''(z_0)}{\Phi'(z_0)} +2\biggr|\leqslant 4,\quad
    \text{or } |y_0\Phi''(z_0)|\leqslant 3|\Phi'(z_0)|.\]
  
  The last inequality of the lemma is an easy consequence of the former two.
\end{proof}

\begin{corollary}
  If $z=x+\mathrm{i}y\in\mathbb{C}_-$, where $x\in\mathbb{R}$, $y<0$, 
  then
  \[|y\Phi_-'(z)|
    \leqslant 2d(\Phi_-(z))
    \leqslant 4|y\Phi_-'(z)|,\]
  and
  \[|y\Phi_-''(z)|\leqslant 3|\Phi_-'(z)|.\]
  Consequently,
  \[|y^2\Phi_-''(z)|\leqslant 6d(\Phi_-(z)).\]
\end{corollary}

Let $\{\gamma_{n\pm}= \gamma_{n\pm}'\cup [-a_n,a_n]\}$ be two series of 
rectifiable simple Jordan curves, which will enventually surround 
any compact subset of $\overline{\mathbb{C}_{\pm}}$. 
Here, $\gamma_{n\pm}'\subset \mathbb{C}_{\pm}$, $a_n>0$ and 
$a_n\to\infty$ as $n\to\infty$, $[-a_n,a_n]$ is a straight segment 
on the real axis. Denote the area measure on $\mathbb{C}$ 
as $\mathrm{d}\lambda$, and measure $|y|\,\mathrm{d}\lambda(z)$ 
as $\mathrm{d}\mu(z)$ for simplicity. Here, $y=\mathrm{Im}\,z$.

\begin{lemma}\label{lem-170727-1050}
  If $H_1$, $H_2$ are two holomorphic functions on $\mathbb{C}_+$, 
  which are both continuous to the boundary, 
  $z= x+ \mathrm{i}y\in\mathbb{C}_+$, and
  
  {\rm (1)} $\int_{\gamma_{n+}'} H_1\overline{H_2}\mathrm{d}z\to 0$ 
  as $n\to\infty$; 
  
  {\rm (2)} $H_1\overline{H_2'}y\to 0$ as $|z|\to\infty$ 
  and either $x$ or $y$ is fixed.\\
  Then
  \begin{equation}\label{equ-170727-1100}
    \int_{\mathbb{R}} H_1\overline{H_2}\,\mathrm{d}x
    = 4\iint_{\mathbb{C}_+} H_1'\overline{H_2'}\,\mathrm{d}\mu(z) 
    = \iint_{\mathbb{C}_+} \Delta(H_1\overline{H_2})\,\mathrm{d}\mu(z),
  \end{equation}
  where $\Delta=4\frac{\partial^2}{\partial z\partial\overline{z}}$ is 
  the Laplace operator.
\end{lemma}

\begin{proof}
  We will write $\gamma_{n+}$ as $\gamma_n= \gamma_n'\cup [-a_n,a_n]$, 
  and denote the domain which $\gamma_n$ surrounds as $D_n$, for fixed $n$. 
  We have, by Green's theorem~\cite{De10},
  \[\int_{\gamma_n} H_1\overline{H_2}\,\mathrm{d}z
    = 2\mathrm{i}\iint_{D_n} \frac{\partial}{\partial\overline{z}}
        (H_1\overline{H_2})\,\mathrm{d}\lambda(z).\]
  Let $n\to\infty$, then by condition~(1) and the definition of $\gamma_n$, 
  the above equation becomes
  \begin{equation}\label{equ-170727-1220}
    \int_{\mathbb{R}} H_1\overline{H_2}\,\mathrm{d}x
    = 2\mathrm{i}\iint_{\mathbb{C}_+} \frac{\partial}{\partial\overline{z}}
        (H_1\overline{H_2})\,\mathrm{d}\lambda(z).
  \end{equation}
  
  Since $H_1\overline{H_2'}y\to 0$ as $y\to\infty$ and $x$ fixed, 
  $\frac{\partial f}{\partial y}= \mathrm{i}f'$ if $f$ is holomorphic, 
  then 
  \begin{align*}
    \iint_{\mathbb{C}_+} \frac{\partial}{\partial\overline{z}}
      (H_1\overline{H_2})\,\mathrm{d}\lambda(z)
    &= \int_{-\infty}^{+\infty}\!\!\int_0^{+\infty} H_1\overline{H_2'}
         \,\mathrm{d}y\,\mathrm{d}x                          \\
    &= \int_{-\infty}^{+\infty} \bigg(H_1\overline{H_2'}y
         \Big|_{y=0}^{y=+\infty}- \int_0^{+\infty} 
           y\frac{\partial}{\partial y}(H_1\overline{H_2'})\,\mathrm{d}y
         \bigg)\mathrm{d}x                                     \\
    &= -\iint_{\mathbb{C}_+} y(\mathrm{i}H_1'\overline{H_2'}
         + H_1\overline{\mathrm{i}H_2''})\,\mathrm{d}\lambda(z)       \\
    &= -\mathrm{i} \iint_{\mathbb{C}_+} y(H_1'\overline{H_2'}
         - H_1\overline{H_2''})\,\mathrm{d}\lambda(z),
  \end{align*}
  and since $H_1\overline{H_2'}y\to 0$ as $|x|\to\infty$ and $y$ fixed, 
  $\frac{\partial f}{\partial x}= f'$ if $f$ is holomorphic,  
  \begin{align*}
    \iint_{\mathbb{C}_+} yH_1\overline{H_2''}\,\mathrm{d}\lambda(z)
    &= \int_0^{+\infty}\!\!\int_{-\infty}^{+\infty} 
         yH_1\frac{\partial}{\partial x}\overline{H_2'}
         \,\mathrm{d}x\,\mathrm{d}y                     \\
    &= \int_0^{+\infty} \bigg(yH_1\overline{H_2'}\Big|_{x=-\infty}^{x=+\infty}
         - \int_{-\infty}^{+\infty} \overline{H_2'}
             \frac{\partial}{\partial x}(yH_1)\,\mathrm{d}x
         \bigg)\mathrm{d}y                          \\
    &= -\iint_{\mathbb{C}_+} yH_1'\overline{H_2'}\,\mathrm{d}\lambda(z),
  \end{align*}
  thus
  \[\iint_{\mathbb{C}_+} \frac{\partial}{\partial\overline{z}}
      (H_1\overline{H_2})\,\mathrm{d}\lambda(z)
    = -2\mathrm{i}\iint_{\mathbb{C}_+} yH_1'\overline{H_2'}
        \,\mathrm{d}\lambda(z).\]
  Together with \eqref{equ-170727-1220}, we get that
  \[\int_{\mathbb{R}} H_1\overline{H_2}\,\mathrm{d}x
    = 4\iint_{\mathbb{C}_+} H_1'\overline{H_2'}y\,\mathrm{d}\lambda(z)
    = 4\iint_{\mathbb{C}_+} H_1'\overline{H_2'}\,\mathrm{d}\mu(z).\]
    
  Since $\Delta$ is the Laplace operator, 
  \[\Delta(H_1\overline{H_2})
    = 4\frac{\partial^2}{\partial z\partial\overline{z}}(H_1\overline{H_2})
    = 4H_1'\overline{H_2'},\]
  and the second equation of \eqref{equ-170727-1100} is obvious.
\end{proof}

\begin{corollary}
  If $H_1$, $H_2$ are two holomorphic functions on $\mathbb{C}_-$, 
  which are both continuous to the boundary, 
  $z= x+\mathrm{i}y\in\mathbb{C}_-$, and
  
  {\rm (1)} $\int_{\gamma_{n-}'} H_1\overline{H_2}\mathrm{d}z\to 0$ 
  as $n\to\infty$; 
  
  {\rm (2)} $H_1\overline{H_2'}y\to 0$ as $|z|\to\infty$ 
  and either $x$ or $y$ is fixed.\\
  Then
  \begin{equation}
    \int_{\mathbb{R}} H_1\overline{H_2}\,\mathrm{d}x
    = 4\iint_{\mathbb{C}_-} H_1'\overline{H_2'}\,\mathrm{d}\mu(z) 
    = \iint_{\mathbb{C}_-} \Delta(H_1\overline{H_2})\,\mathrm{d}\mu(z),
  \end{equation}
  where $\Delta=4\frac{\partial^2}{\partial z\partial\overline{z}}$ is 
  the Laplace operator.
\end{corollary}

Define $\mathcal{T}(D)$ as the holomorphic function space
on domain~$D$, which satisfies that if $F\in\mathcal{T}(D)$, 
then there exist two constants $A$, $r>0$, depending on $F$ only, 
such that
\[|F(z)|\leqslant \frac{A}{|z|},\quad
  \text{and } |F'(z)|\leqslant \frac{A}{|z|^2},\quad
  \text{for } |z|>r \text{ and } z\in D.\]

\begin{corollary}\label{cor-170727-1600}
  If $F\in\mathcal{T}(\Omega_+)$ is continuous to the boundary~$\Gamma$, then 
  \[\int_{\mathbb{R}} |F(\Phi_+)|^2\Phi_+'\,\mathrm{d}x
    = \int_{\mathbb{C}_+} \Delta(|F(\Phi_+)|^2\Phi_+') \,\mathrm{d}\mu(z).\]  
\end{corollary}

\begin{proof}
  Suppose $A$, $r$ are the two constants related to $F$ 
  in the definition of $\mathcal{T}(\Omega_+)$. Let $R>r$, 
  $l= R\sqrt{1+M^2}+|a(0)|$, and
  \[E(R)= \{u+\mathrm{i}v\colon |u|<R, |v|<l\}\cap \Omega_+,\]
  then $E(R)\neq\emptyset$, and $\partial E(R)$ consists of a curve segment 
  and three straight segments. Denote $\partial E(R)$ as $BCC'B'$, where
  \begin{align*}
    BC&=\{u+\mathrm{i}a(u)\colon |u|\leqslant R\},\quad
    CC'=\{R+\mathrm{i}v\colon v\in[a(R),l]\},                  \\
    C'B'&=\{u+\mathrm{i}l\colon |u|\leqslant R\},\quad
    B'B=\{-R+\mathrm{i}v\colon v\in[a(-R),l]\}.
  \end{align*}
  
  We then consider $\gamma_R=\partial\Psi(E(R))\subset \mathbb{C}_+$, 
  and let $H_1=F(\Phi)\Phi'$, $H_2=F(\Phi)$. In order to invoke 
  Lemma~\ref{lem-170727-1050}, it is sufficient to varify the two conditions 
  in that lemma. Let $R\to\infty$, we have, by the definition of 
  $F\in\mathcal{T}_+$,
  \begin{align*}
    \Big|\int_{\Psi(CC')} H_1\overline{H_2}\,\mathrm{d}z\Big|
    &\leqslant \int_{\Psi(CC')} |F(\Phi)|^2 |\Phi'\,\mathrm{d}z| 
     = \int_{CC'} |F(w)|^2 |\mathrm{d}w|                     \\
    &\leqslant \int_{CC'} \frac{A^2}{|w|^2} |\mathrm{d}w|         
     \leqslant \int_{-\infty}^{+\infty} \frac{A^2\,\mathrm{d}v}{R^2+v^2} \\
    &= \frac{A^2\pi}{R} \to 0,
  \end{align*}
  and
  \begin{align*}
    \Big|\int_{\Psi(C'B')} H_1\overline{H_2}\,\mathrm{d}z\Big|
    &\leqslant \int_{\Psi(C'B')} |F(\Phi)|^2 |\Phi'\,\mathrm{d}z|  
     \leqslant \int_{C'B'} \frac{A^2}{|w|^2} |\mathrm{d}w|         \\
    &\leqslant \int_{-\infty}^{+\infty} \frac{A^2\,\mathrm{d}u}{u^2+l^2} 
     = \frac{A^2\pi}{l} \to 0,
  \end{align*}
  since $l\to\infty$ as $R\to\infty$. We also have 
  $\int_{\Psi(B'B)} H_1\overline{H_2}\,\mathrm{d}z\to 0$ 
  by applying the same method.
  
  For the second condition in Lemma~\ref{lem-170727-1050}, since 
  $H_2'=F'(\Phi)\Phi'$ and $\Phi(\infty)=\infty$, we have, 
  by Lemma~\ref{lem-170727-1550},
  \begin{align*}
    |H_1\overline{H_2'}y|
    &= |F(\Phi)F'(\Phi)(\Phi')^2 y|                \\
    &\leqslant \frac{A^2y|\Phi'|^2}{|\Phi|^3}
     \leqslant \frac{A^2}{|\Phi|^3 y}\cdot |2d(\Phi)|^2       \\
    &\leqslant \frac{4A^2}{|\Phi|^3 y}\cdot (|\Phi|+|a(0)|)^2 
      \to 0,
  \end{align*}
  if $z=x+\mathrm{i}y\to \infty$ with either $x\in\mathbb{R}$ 
  or $y>0$ fixed.
  
  Thus Lemma~\ref{lem-170727-1050} implies that
  \[\int_{\mathbb{R}} H_1\overline{H_2}\,\mathrm{d}x
    = \iint_{\mathbb{C}_+} \Delta(H_1\overline{H_2})\,\mathrm{d}\mu(z),\]
  which is the desired equation after replacing $H_1$ with $F(\Phi)\Phi'$ 
  and $H_2$ with $F(\Phi)$.
\end{proof}

\begin{corollary}
  If $F\in\mathcal{T}(\Omega_-)$ is continuous to the boundary~$\Gamma$, then 
  \[\int_{\mathbb{R}} |F(\Phi_-)|^2\Phi_-'\,\mathrm{d}x
    = \int_{\mathbb{C}_-} \Delta(|F(\Phi_-)|^2\Phi_-')
      \,\mathrm{d}\mu(z).\]  
\end{corollary}

The following lemma has been proved in~\cite{DL171}.

\begin{lemma}\label{lem-170522-1205}
  For the conformal representation $\Phi_+\colon \mathbb{C}_+\to\Omega_+$,
  we have
  \[\mathrm{Re}\,\Phi_+'(z)>0, \quad \text{and } 
    \lvert\mathrm{Im}\,\Phi_+'(z)\rvert\leqslant M\mathrm{Re}\,\Phi_+'(z),\]
  or
  \[|\arg\Phi_+'(z)|\leqslant \arctan M< \frac{\pi}2,\]
  for all $z\in\mathbb{C}_+$.
\end{lemma}

\begin{corollary}
  For the conformal representation $\Phi_-\colon \mathbb{C}_-\to\Omega_-$,
  we have
  \[\mathrm{Re}\,\Phi_-'(z)>0, \quad \text{and } 
    \lvert\mathrm{Im}\,\Phi_-'(z)\rvert\leqslant M\mathrm{Re}\,\Phi_-'(z),\]
  or
  \[|\arg\Phi_-'(z)|\leqslant \arctan M< \frac{\pi}2,\]
  for all $z\in\mathbb{C}_-$.
\end{corollary}

Since $\Phi_\pm'\neq 0$ on $\Omega_\pm$, we could let 
$\Phi_\pm'= \mathrm{e}^{V_\pm}$ and $D_\pm=\mathrm{e}^{\mathrm{i}V_\pm}$, 
then
\[\Phi_\pm''= V_\pm'\mathrm{e}^{V_\pm}= V_\pm'\Phi_\pm',\]
and
\[D_\pm'= \mathrm{i}V_\pm'\mathrm{e}^{\mathrm{i}V_\pm}
  = \mathrm{i}V_\pm'D_\pm
  = \mathrm{i}D_\pm\Phi_\pm''(\Phi_\pm')^{-1}.\]
Denote $\theta_0=\arctan M$, by Lemma~\ref{lem-170522-1205},
\[|\mathrm{Im}\,V_\pm|= |\arg\Phi_\pm'(z)|\leqslant \theta_0,\]
and
\[|D_\pm|= \mathrm{e}^{-\mathrm{Im}\,V_\pm}
  \in [\mathrm{e}^{-\theta_0}, \mathrm{e}^{\theta_0}],\]
then
\[|\Phi_\pm''|= \Big|\frac{D_\pm'}{\mathrm{i}D_\pm}\cdot \Phi_\pm'\Big|
  \leqslant \mathrm{e}^{\theta_0} |D_\pm'\Phi_\pm'|.\]

\begin{corollary}\label{cor-170727-1610}
  If $F\in\mathcal{T}(\Omega_+)$ is continuous to the boundary~$\Gamma$, 
  $D_+$ is defined as above, and let $H= F(\Phi_+)(\Phi_+')^{\frac12}$, 
  then 
  \[\lVert H\rVert_{L^2(\mathbb{R},\mathrm{d}x)}
    = 2\lVert H'\rVert_{L^2(\mathbb{C}_+,\mathrm{d}\mu)},\]
  and
  \[\lVert HD_+\rVert_{L^2(\mathbb{R},\mathrm{d}x)}
    = 2\lVert (HD_+)'\rVert_{L^2(\mathbb{C}_+,\mathrm{d}\mu)}.\]
  Consequently,
  \[\lVert HD_+'\rVert_{L^2(\mathbb{R},\mathrm{d}x)}
    \leqslant \mathrm{e}^{\theta_0}
        \lVert H\rVert_{L^2(\mathbb{R},\mathrm{d}x)}.\]
\end{corollary}

\begin{proof}
  Suppose $F\in\mathcal{T}(\Omega_+)$ and denote $\Phi_+$ as $\Phi$,  
  we consider the same $\partial\Psi(E(R))$ as in 
  Corollary~\ref{cor-170727-1600}, 
  and let $H_1= H_2= H= F(\Phi)(\Phi')^{\frac12}$, 
  then the first condition in Lemma~\ref{lem-170727-1050} is verified 
  in the same way as in that corollary. 
  For the second condition, since
  \[H_2'= H'
    = F'(\Phi)(\Phi')^{\frac32}+ \frac12 F(\Phi)(\Phi')^{-\frac12}\Phi'',\]
  we have, by Lemma~\ref{lem-170727-1550} and 
  $d(\Phi)\leqslant |\Phi|+|a(0)|$,
  \begin{align*}
    |H_1\overline{H_2'}y|
    &= y|F(\Phi)(\Phi')^{\frac12}|\cdot
        \Big|F'(\Phi)(\Phi')^{\frac32}+ 
        \frac12 F(\Phi)(\Phi')^{-\frac12}\Phi''\Big|             \\
    &\leqslant y\frac{A|\Phi'|^{\frac12}}{|\Phi|}\bigg(
        \frac{A|\Phi'|^{\frac32}}{|\Phi|^2}+ 
          \frac{A|\Phi''|}{2|\Phi|\cdot |\Phi'|^{\frac12}}\bigg)     \\
    &= \frac{yA^2}{|\Phi|^3} \Big(|\Phi'|^2+\frac12|\Phi\Phi''|\Big)      \\
    &= \frac{A^2}{y|\Phi|^3} \Big(|y\Phi'|^2+\frac12|\Phi\cdot y^2\Phi''|\Big)\\
    &\leqslant \frac{A^2}{y|\Phi|^3} \big(|2d(\Phi)|^2+ 3|\Phi|d(\Phi)\big)\\
    &= \frac{A^2}{y|\Phi|^3}\cdot d(\Phi)\big(4d(\Phi)+3|\Phi|\big)     \\
    &\leqslant \frac{A^2}{y|\Phi|^3}\big(|\Phi|+|a(0)|\big)
        \big(7|\Phi|+3|a(0)|\big) \to 0,
  \end{align*}
  as $z=x+\mathrm{i}y\to \infty$ with either $x\in\mathbb{R}$ 
  or $y>0$ fixed. Then, by Lemma~\ref{lem-170727-1050},
  \[\int_{\mathbb{R}} |H|^2\,\mathrm{d}x
    = 4\iint_{\mathbb{C}_+} |H'|^2\,\mathrm{d}\mu(z),\]
  which is 
  \[\lVert H\rVert_{L^2(\mathbb{R},\mathrm{d}x)}
    = 2\lVert H'\rVert_{L^2(\mathbb{C}_+,\mathrm{d}\mu)}.\]
    
  Next, denote $D_+$ as $D$ and let $H_3= H_4= HD$. 
  Since $|D|\leqslant\mathrm{e}^{\theta_0}$, the first condition in 
  Lemma~\ref{lem-170727-1050} could be easily verified. We now turn to 
  the second condition. Since $D'= \mathrm{i}D\Phi''(\Phi')^{-1}$, and
  \begin{align*}
    H_4'
    &= H'D+ HD'                      \\
    &= F'(\Phi)(\Phi')^{\frac32}D+ \frac12 F(\Phi)(\Phi')^{-\frac12}\Phi''D
       + F(\Phi)(\Phi')^{\frac12}\cdot \mathrm{i}D\Phi''(\Phi')^{-1} \\
    &= F'(\Phi)(\Phi')^{\frac32}D
       + \Big(\frac12+\mathrm{i}\Big)F(\Phi)(\Phi')^{-\frac12}\Phi''D,
  \end{align*}
  we have, by $H_3= F(\Phi)(\Phi')^{\frac12}D$,
  \begin{align*}
    |H_3\overline{H_4}y|
    &\leqslant \mathrm{e}^{2\theta_0} y|F(\Phi)(\Phi')^{\frac12}|
        \cdot \Big(|F'(\Phi)(\Phi')^{\frac32}|
               + \frac32|F(\Phi)(\Phi')^{-\frac12}\Phi''|\Big)      \\
    &\leqslant \mathrm{e}^{2\theta_0} y\cdot \frac{A|\Phi'|^{\frac12}}{|\Phi|}
        \cdot \bigg(\frac{A|\Phi'|^{\frac32}}{|\Phi|^2}
               + \frac{3A|\Phi''|}{2|\Phi|\cdot|\Phi'|^{\frac12}}\bigg)  \\
    &\leqslant \frac{\mathrm{e}^{2\theta_0} A^2}{y|\Phi|^3} 
         \big(|2d(\Phi)^2|+ 9d(\Phi)|\Phi|\big)                          \\
    &\leqslant \frac{\mathrm{e}^{2\theta_0} A^2}{y|\Phi|^3} 
         \big(13|\Phi|+4|a(0)|\big)\big(|\Phi|+|a(0)|\big)
     \to 0,
  \end{align*}
  as $z=x+\mathrm{i}y\to \infty$ with either $x\in\mathbb{R}$ 
  or $y>0$ fixed. Then,
  \[\int_{\mathbb{R}} |HD|^2\,\mathrm{d}x
    = 4\iint_{\mathbb{C}_+} |(HD)'|^2\,\mathrm{d}\mu(z),\]
  or equivalently,
  \[\lVert HD\rVert_{L^2(\mathbb{R},\mathrm{d}x)}
    = 2\lVert (HD)'\rVert_{L^2(\mathbb{C}_+,\mathrm{d}\mu)},\]
  which finishes the equation part of the corollary.
  
  Since $|D|\leqslant\mathrm{e}^{\theta_0}$, it follows that,
  \[\lVert HD\rVert_{L^2(\mathbb{R},\mathrm{d}x)}
    \leqslant \mathrm{e}^{\theta_0} 
        \lVert H\rVert_{L^2(\mathbb{R},\mathrm{d}x)},\]
  and
  \begin{align*}
    \lVert (HD)'\rVert_{L^2(\mathbb{C}_+,\mathrm{d}\mu)}
    &= \lVert HD'+H'D\rVert_{L^2(\mathbb{C}_+,\mathrm{d}\mu)}     \\
    &\geqslant \lVert HD'\rVert_{L^2(\mathbb{C}_+,\mathrm{d}\mu)}
        - \lVert H'D\rVert_{L^2(\mathbb{C}_+,\mathrm{d}\mu)}      \\
    &\geqslant \lVert HD'\rVert_{L^2(\mathbb{C}_+,\mathrm{d}\mu)}
        - \mathrm{e}^{\theta_0}
          \lVert H'\rVert_{L^2(\mathbb{C}_+,\mathrm{d}\mu)}      \\
    &= \lVert HD'\rVert_{L^2(\mathbb{C}_+,\mathrm{d}\mu)}
        - \frac12\mathrm{e}^{\theta_0}
          \lVert H\rVert_{L^2(\mathbb{R},\mathrm{d}x)},
  \end{align*}
  then
  \[\mathrm{e}^{\theta_0} \lVert H\rVert_{L^2(\mathbb{R},\mathrm{d}x)}
    \geqslant 2\lVert HD'\rVert_{L^2(\mathbb{C}_+,\mathrm{d}\mu)}
       - \mathrm{e}^{\theta_0}\lVert H\rVert_{L^2(\mathbb{R},\mathrm{d}x)},\]
  or 
  \[\mathrm{e}^{\theta_0} \lVert H\rVert_{L^2(\mathbb{R},\mathrm{d}x)}
    \geqslant \lVert HD'\rVert_{L^2(\mathbb{C}_+,\mathrm{d}\mu)},\]
  which proves the corollary.
\end{proof}

\begin{corollary}
  If $F\in\mathcal{T}(\Omega_-)$ is continuous to the boundary~$\Gamma$, 
  $D_-$ is defined as above, and let $H= F(\Phi_-)(\Phi_-')^{\frac12}$, 
  then 
  \[\lVert H\rVert_{L^2(\mathbb{R},\mathrm{d}x)}
    = 2\lVert H'\rVert_{L^2(\mathbb{C}_-,\mathrm{d}\mu)},\]
  and
  \[\lVert HD_-\rVert_{L^2(\mathbb{R},\mathrm{d}x)}
    = 2\lVert (HD_-)'\rVert_{L^2(\mathbb{C}_-,\mathrm{d}\mu)}.\]
  Consequently,
  \[\lVert HD_-'\rVert_{L^2(\mathbb{R},\mathrm{d}x)}
    \leqslant \mathrm{e}^{\theta_0}
        \lVert H\rVert_{L^2(\mathbb{R},\mathrm{d}x)}.\]
\end{corollary}

\section{Proof of the Main Theorems}

Denote the measure $d(w)\,\mathrm{d}\lambda(w)$ on $\mathbb{C}$ 
as $\mathrm{d}\nu(w)$ and consider the function spaces 
$L^2(\Omega_\pm,\mathrm{d}\nu)$, thus, 
if $F\in L^2(\Omega_\pm,\mathrm{d}\nu)$, then
\begin{align*}
  \lVert F\rVert_{L^2(\Omega_\pm,\mathrm{d}\nu)}^2
  &= \iint_{\Omega_\pm} |F|^2\,\mathrm{d}\nu(w)         \\
  &= \iint_{\Omega_\pm} |F|^2 d(w)\,\mathrm{d}\lambda(w)
   < \infty.
\end{align*}

\begin{theorem}\label{thm-170727-1750}
  If $F\in\mathcal{T}(\Omega_+)$ and $\tau>0$ is fixed, then 
  \[\lVert F(\cdot+\mathrm{i}\tau)\rVert_{L^2(\Gamma,|\mathrm{d}\zeta|)}
    \leqslant C\sqrt{1+M^2}\lVert F'\rVert_{L^2(\Omega_+,\mathrm{d}\nu)}.\]
  If, furthermore, $F$ is continuous to the boundary, then we also have
  \[\lVert F\rVert_{L^2(\Gamma,|\mathrm{d}\zeta|)}
    \leqslant C\sqrt{1+M^2}\lVert F'\rVert_{L^2(\Omega_+,\mathrm{d}\nu)}.\]
  Here, $C=7\mathrm{e}^{2\theta_0}$ and $M=\lVert a'\rVert_\infty$.
\end{theorem}

\begin{proof}
  We first assume $F\in\mathcal{T}(\Omega_+)$ is continuous to the boundary 
  and denote $\Phi_+$ as $\Phi$. Since 
  $\Gamma=\{\zeta(u)=u+\mathrm{i}a(u)\colon u\in\mathbb{R}\}$ 
  is a Lipschitz curve and, by Lemma~\ref{lem-170522-1205}, 
  $\lvert\mathrm{Im}\,\Phi'(z)\rvert\leqslant M\mathrm{Re}\,\Phi'(z)$, 
  we have
  \begin{align*}
    \lVert F\rVert_{L^2(\Gamma,|\mathrm{d}\zeta|)}^2
    &= \int_\Gamma |F|^2|\mathrm{d}\zeta|            
     = \int_{\mathbb{R}} |F(\Phi)|^2 |\Phi'|\,\mathrm{d}x  \\
    &\leqslant \sqrt{1+M^2} \int_{\mathbb{R}} |F(\Phi)|^2 
        \mathrm{Re}\,\Phi'\,\mathrm{d}x                    \\
    &\leqslant \sqrt{1+M^2} \Big|\int_{\mathbb{R}} |F(\Phi)|^2 
         \Phi'\,\mathrm{d}x\Big|.
  \end{align*}
  Denote $\int_{\mathbb{R}} |F(\Phi)|^2 |\Phi'|\,\mathrm{d}x$ as $I_1$, 
  $\sqrt{1+M^2}$ as $M_1$, $F(\Phi)$ as $H$, 
  then, by Corollary~\ref{cor-170727-1600}, 
  \begin{equation}\label{equ-170727-2010}
    \begin{aligned}
    I_1
    &\leqslant M_1\Big|\int_{\mathbb{C}_+} \Delta(|H|^2\Phi')
          \,\mathrm{d}\mu(z)\Big|                                 \\
    &= 4M_1\Big|\int_{\mathbb{C}_+} (|H'|^2\Phi'+ H\overline{H'}\Phi'')
          \,\mathrm{d}\mu(z)\Big|                                 \\
    &= 4M_1\bigg(\int_{\mathbb{C}_+} |H'|^2|\Phi'|\,\mathrm{d}\mu(z)
         + \int_{\mathbb{C}_+} |H\overline{H'}\Phi''|\,\mathrm{d}\mu(z)\bigg).
    \end{aligned}
  \end{equation}
  We denote the first integral right above as $I_2$. It follows that, 
  by Lemma~\ref{lem-170727-1550}, 
  \begin{equation}\label{equ-170727-2020}
    \begin{aligned}
    I_2
    &= \lVert H'(\Phi')^{\frac12}\rVert_{L^2(\mathbb{C}_+,\mathrm{d}\mu)}^2 \\
    &= \int_{\mathbb{C}_+} |F'(\Phi)|^2|\Phi'|^3 y\,\mathrm{d}\lambda(z)  \\
    &\leqslant 2\int_{\mathbb{C}_+} |F'(\Phi)|^2|\Phi'|^2 
         d(\Phi)\,\mathrm{d}\lambda(z)                              \\
    &= 2\int_{\Omega_+} |F'(w)|^2 d(w)\,\mathrm{d}\lambda(w)         \\
    &= 2\lVert F'\rVert_{L^2(\Omega_+,\mathrm{d}\nu)}^2.
  \end{aligned}
  \end{equation}
  
  Since $\Phi'=\mathrm{e}^V$ and $\Phi''= V'\Phi'$, H\"{o}lder's inequality 
  implies that 
  \begin{align*}
    \int_{\mathbb{C}_+} |H\overline{H'}\Phi''|\,\mathrm{d}\mu(z)
    &= \int_{\mathbb{C}_+} |HV'(\Phi')^{\frac12}|
         \cdot |\overline{H'}(\Phi')^{\frac12}|\,\mathrm{d}\mu(z)   \\
    &\leqslant \lVert HV'(\Phi')^{\frac12}
                 \rVert_{L^2(\mathbb{C}_+,\mathrm{d}\mu)}
         \cdot \lVert \overline{H'}(\Phi')^{\frac12}
                 \rVert_{L^2(\mathbb{C}_+,\mathrm{d}\mu)}           \\
    &= \lVert HV'(\Phi')^{\frac12}\rVert_{L^2(\mathbb{C}_+,\mathrm{d}\mu)}
         \cdot I_2^{\frac12}.
  \end{align*}
  Notice that $D=\mathrm{e}^{\mathrm{i}V}$ and 
  $|V'|=|D'D^{-1}|\leqslant \mathrm{e}^{\theta_0}|D'|$, then 
  \begin{align*}
    \lVert HV'(\Phi')^{\frac12}\rVert_{L^2(\mathbb{C}_+,\mathrm{d}\mu)}
    &\leqslant \mathrm{e}^{\theta_0} \lVert H(\Phi')^{\frac12}D'
                  \rVert_{L^2(\mathbb{C}_+,\mathrm{d}\mu)}           \\
    &\leqslant \mathrm{e}^{2\theta_0} \lVert H(\Phi')^{\frac12}
                  \rVert_{L^2(\mathbb{R},\mathrm{d}x)}              \\
    &= \mathrm{e}^{2\theta_0} I_1^{\frac12},
  \end{align*}
  by Corollary~\ref{cor-170727-1610}, since 
  $H(\Phi')^{\frac12}= F(\Phi)(\Phi')^{\frac12}$.
  
  Thus, inequality~\eqref{equ-170727-2010} becomes
  \[I_1\leqslant 4M_1(I_2+ I_2^{\frac12}
            \cdot \mathrm{e}^{2\theta_0} I_1^{\frac12}),\]
  and we rewrite it as
  \[I_1-8M_1\mathrm{e}^{2\theta_0} I_1^{\frac12}I_2^{\frac12}
    \leqslant 8M_1 I_2- I_1,\]
  or
  \[(I_1^{\frac12}-4M_1\mathrm{e}^{2\theta_0} I_2^{\frac12})^2
    \leqslant (8M_1+16M_1^2\mathrm{e}^{4\theta_0})I_2- I_1,\]
  and then 
  \[I_1\leqslant 8M_1(1+2M_1\mathrm{e}^{4\theta_0})I_2
    \leqslant 24\mathrm{e}^{4\theta_0} M_1^2 I_2,\]
  as $M_1 \mathrm{e}^{4\theta_0}
      = \mathrm{e}^{4\theta_0}\sqrt{1+M^2}\geqslant 1$. 
  Together with~\eqref{equ-170727-2020}, we have 
  \[\lVert F\rVert_{L^2(\Gamma,|\mathrm{d}\zeta|)}^2
    = I_1 
    \leqslant 48\mathrm{e}^{4\theta_0} M_1^2 
          \lVert F'\rVert_{L^2(\Omega_+,\mathrm{d}\nu)}^2,\]
  and
  \[\lVert F\rVert_{L^2(\Gamma,|\mathrm{d}\zeta|)}
    \leqslant 7\mathrm{e}^{2\theta_0} M_1\lVert F'\rVert_{L^2(\Omega_+,\mathrm{d}\nu)}
    = C\sqrt{1+M^2}\lVert F'\rVert_{L^2(\Omega_+,\mathrm{d}\nu)},\]
  where $C=7\mathrm{e}^{2\theta_0}$.
  
  For the general case of $F\in\mathcal{T}(\Omega_+)$, fix $\tau>0$ and define 
  $G(w)=F(w+\mathrm{i}\tau)$ for $w\in\Omega_+$, 
  then $G\in\mathcal{T}(\Omega_+)$
  and is continous to the boundary. By what we have proved,
  \[\lVert F(\cdot+\mathrm{i}\tau)\rVert_{L^2(\Gamma,|\mathrm{d}\zeta|)}
    =\lVert G\rVert_{L^2(\Gamma,|\mathrm{d}\zeta|)}
    \leqslant C\sqrt{1+M^2}\lVert G'\rVert_{L^2(\Omega_+,\mathrm{d}\nu)}.\] 
  Since $G'(w)=F'(w+\mathrm{i}\tau)$ and $d(w)\leqslant d(w+\mathrm{i}\tau)$ 
  for $w\in\Omega_+$,
  \begin{align*}
    \lVert G'\rVert_{L^2(\Omega_+,\mathrm{d}\nu)}^2
    &= \iint_{\Omega_+} |F'(w+\mathrm{i}\tau)|^2 d(w)\,\mathrm{d}\lambda(w) \\
    &\leqslant \iint_{\Omega_+} |F'(w+\mathrm{i}\tau)|^2 
          d(w+\mathrm{i}\tau)\,\mathrm{d}\lambda(w)           \\
    &= \iint_{\Omega_{+}+\mathrm{i}\tau} |F'(w)|^2 
          d(w)\,\mathrm{d}\lambda(w)                          \\
    &\leqslant \iint_{\Omega_+} |F'(w)|^2 d(w)\,\mathrm{d}\lambda(w)   \\
    &= \lVert F'\rVert_{L^2(\Omega_+,\mathrm{d}\nu)}^2,
  \end{align*}
  where $\Omega_{+}+\mathrm{i}\tau= \{w+\mathrm{i}\tau\colon w\in\Omega_{+}\}
    \subset \Omega_+$ for $\tau>0$. Thus
  \[\lVert F(\cdot+\mathrm{i}\tau)\rVert_{L^2(\Gamma,|\mathrm{d}\zeta|)}
      \leqslant C\sqrt{1+M^2}\lVert F'\rVert_{L^2(\Omega_+,\mathrm{d}\nu)},\]
  and the theorem is proved.
\end{proof}

\begin{corollary}\label{cor-170729-2350}
  If $F\in\mathcal{T}(\Omega_-)$ and $\tau>0$ is fixed, then 
  \[\lVert F(\cdot-\mathrm{i}\tau)\rVert_{L^2(\Gamma,|\mathrm{d}\zeta|)}
    \leqslant C\sqrt{1+M^2}\lVert F'\rVert_{L^2(\Omega_-,\mathrm{d}\nu)}.\]
  If, furthermore, $F$ is continuous to the boundary, then we also have
  \[\lVert F\rVert_{L^2(\Gamma,|\mathrm{d}\zeta|)}
    \leqslant C\sqrt{1+M^2}\lVert F'\rVert_{L^2(\Omega_-,\mathrm{d}\nu)}.\]
  Here, $C=7\mathrm{e}^{2\theta_0}$ and $M=\lVert a'\rVert_\infty$.
\end{corollary}

\begin{theorem}\label{thm-170727-2120}
  Let $f\in L^2(\Omega_+,\mathrm{d}\nu)$ be compactly supported, and define, 
  for $w_2\in\overline{\Omega_-}$,
  \[Tf(w_2)= \iint_{\Omega_+} \frac{f(w_1)d(w_1)}{(w_1-w_2)^2} 
      \,\mathrm{d}\lambda(w_1)
    = \iint_{\Omega_+} \frac{f(w_1)\,\mathrm{d}\nu(w_1)}{(w_1-w_2)^2}.\]
  Then, $\lVert Tf\rVert_{L^2(\Gamma,|\mathrm{d}\zeta|)}
    \leqslant C\sqrt{1+M^2}\lVert f\rVert_{L^2(\Omega_+,\mathrm{d}\nu)}$, 
  where $C=56\pi\mathrm{e}^{2\theta_0}$.
\end{theorem}

\begin{proof}
  Suppose $E=\mathrm{supp}\,f\subset D(0,R)$, where $R>0$. 
  Since $E\subset\Omega_+$ is compact, $Tf$ is holomorphic on a neighborhood 
  of $\overline{\Omega_-}$, thus continuous to the boundary~$\Gamma$. 
  If $|w_2|>2R$, then $|w_2-w_1|>\frac12|w_2|$ for $|w_1|<R$. 
  Since $d(w_1)\leqslant |w_2-w_1|$ for $w_1\in\Omega_+$ and $w_2\in\Omega_-$, 
  we have, by H\"{o}lder's inequality,
  \begin{align*}
    |Tf(w_2)|
    &\leqslant \frac2{|w_2|} \iint_E |f(w_1)|\,\mathrm{d}\lambda(w_1) \\
    &\leqslant \frac2{|w_2|} \bigg(\iint_E |f(w_1)|^2 
          d(w_1)\,\mathrm{d}\lambda(w_1)\bigg)^{\frac12} 
        \bigg(\iint_E \frac{\mathrm{d}\lambda(w_1)}{d(w_1)}\bigg)^{\frac12} \\
    &= \frac2{|w_2|} \lVert f\rVert_{L^2(\Omega_+,\mathrm{d}\nu)}
        \bigg(\iint_E \frac{\mathrm{d}\lambda(w_1)}{d(w_1)}\bigg)^{\frac12} \\
    &= \frac{2A}{|w_2|},
  \end{align*}
  where $A= \lVert f\rVert_{L^2(\Omega_+,\mathrm{d}\nu)}
    (\iint_E d^{-1}\mathrm{d}\lambda)^{\frac12}$,
  and
  \begin{align*}
    |(Tf)'(w_2)|
    &= \Big| 2\iint_{\Omega_+} \frac{f(w_1)d(w_1)}{(w_1-w_2)^3} 
          \,\mathrm{d}\lambda(w_1)\Big|                             \\
    &\leqslant 2\iint_{\Omega_+} \frac{|f(w_1)| d(w_1)}{|w_1-w_2|^3} 
          \,\mathrm{d}\lambda(w_1)                                  \\
    &\leqslant \frac8{|w_2|^2} \iint_E |f(w_1)|\,\mathrm{d}\lambda(w_1) \\
    &= \frac{8A}{|w_2|^2}.
  \end{align*}
  Thus $Tf\in\mathcal{T}(\Omega_-)$ and by Corollary~\ref{cor-170729-2350},
  \begin{equation}\label{equ-170727-2250}
    \lVert Tf\rVert_{L^2(\Gamma,|\mathrm{d}\zeta|)}
    \leqslant C\sqrt{1+M^2}\lVert (Tf)'\rVert_{L^2(\Omega_-,\mathrm{d}\nu)},
  \end{equation}
  where $C=7\mathrm{e}^{2\theta_0}$.
  
  Define an operator $S\colon L^2(\Omega_+,\mathrm{d}\lambda)\to 
    L^2(\Omega_-,\mathrm{d}\lambda)$ by
  \begin{align*}
    SF(w_2)
    &= d(w_2)^{\frac12} \iint_{\Omega_+} 
         \frac{F(w_1)d(w_1)^{\frac12}}{|w_1-w_2|^3}\,\mathrm{d}\lambda(w_1)\\
    &= \iint_{\Omega_+} K(w_1,w_2)F(w_1)\,\mathrm{d}\lambda(w_1),
  \end{align*}
  where $w_2\in\Omega_-$ 
  and $K(w_1,w_2)= d(w_1)^{\frac12} d(w_2)^{\frac12} |w_1-w_2|^{-3}$.
  For $w_2\in\Omega_-$ fixed, since $d(w_1)\leqslant |w_1-w_2|$ and 
  $\Omega_+\subset\mathbb{C}\setminus D(w_2,d(w_2))$, then
  \begin{align*}
    \iint_{\Omega_+} K(w_1,w_2)\,\mathrm{d}\lambda(w_1)
    &\leqslant d(w_2)^{\frac12} \iint_{\mathbb{C}\setminus D(w_2,d(w_2))} 
        |w_1-w_2|^{-\frac52}\,\mathrm{d}\lambda(w_1)                  \\
    &= d(w_2)^{\frac12} \int_0^{2\pi}\!\!\int_{d(w_2)}^{+\infty}
        r^{-\frac52}\cdot r\,\mathrm{d}r\,\mathrm{d}\theta            \\
    &= 2\pi d(w_2)^{\frac12} \int_{d(w_2)}^{+\infty}
            r^{-\frac32}\,\mathrm{d}r                          \\
    &= 2\pi d(w_2)^{\frac12}\cdot 2d(w_2)^{-\frac12}           \\
    &= 4\pi.
  \end{align*}
  The same computation yields that, for fixed $w_1\in\Omega_+$,
  \[\iint_{\Omega_-} K(w_1,w_2)\,\mathrm{d}\lambda(w_2)\leqslant 4\pi.\]
  By Schur's lemma~\cite{Gr08}, $S$ is a bounded operator from 
  $L^2(\Omega_+,\mathrm{d}\lambda)$ to $L^2(\Omega_-,\mathrm{d}\lambda)$, 
  and $\lVert S\rVert\leqslant 4\pi$.
  
  If we let $F(w_1)= f(w_1)d(w_1)^{\frac12}$, then 
  \[\lVert SF\rVert_{L^2(\Omega_-,\mathrm{d}\lambda)}
    \leqslant 4\pi\lVert F\rVert_{L^2(\Omega_+,\mathrm{d}\lambda)}
    = 4\pi\lVert f\rVert_{L^2(\Omega_+,\mathrm{d}\nu)},\]
  and 
  \[|(Tf)'(w_2)|
    \leqslant 2\iint_{\Omega_+} \frac{|f(w_1)|d(w_1)}{|w_1-w_2|^3} 
            \,\mathrm{d}\lambda(w_1) 
    = \frac{2(SF)(w_2)}{d(w_2)^{\frac12}}.\]
  It follows that,
  \begin{align*}
    \lVert (Tf)'\rVert_{L^2(\Omega_-,\mathrm{d}\nu)}
    &\leqslant 2\lVert d^{-\frac12}SF\rVert_{L^2(\Omega_-,\mathrm{d}\nu)} \\
    &= 2\lVert SF\rVert_{L^2(\Omega_-,\mathrm{d}\lambda)}                 \\
    &\leqslant 8\pi\lVert f\rVert_{L^2(\Omega_+,\mathrm{d}\nu)},
  \end{align*}
  thus, by \eqref{equ-170727-2250},
  \begin{align*}
    \lVert Tf\rVert_{L^2(\Gamma,|\mathrm{d}\zeta|)}
    &\leqslant C\sqrt{1+M^2}\cdot 
       8\pi\lVert f\rVert_{L^2(\Omega_+,\mathrm{d}\nu)}       \\
    &\leqslant C'\sqrt{1+M^2}\lVert f\rVert_{L^2(\Omega_+,\mathrm{d}\nu)},
  \end{align*}
  where $C'= 56\pi\mathrm{e}^{2\theta_0}$, and this proves the theorem.
\end{proof}

\begin{corollary}
  Let $f\in L^2(\Omega_-,\mathrm{d}\nu)$ be compactly supported, and define, 
  for $w_1\in\overline{\Omega_+}$,
  \[Tf(w_1)= \iint_{\Omega_-} \frac{f(w_2)d(w_2)}{(w_2-w_1)^2} 
      \,\mathrm{d}\lambda(w_2)
    = \iint_{\Omega_+} \frac{f(w_2)\,\mathrm{d}\nu(w_2)}{(w_2-w_1)^2}.\]
  Then, $\lVert Tf\rVert_{L^2(\Gamma,|\mathrm{d}\zeta|)}
    \leqslant C\sqrt{1+M^2}\lVert f\rVert_{L^2(\Omega_-,\mathrm{d}\nu)}$, 
  where $C=56\pi\mathrm{e}^{2\theta_0}$.
\end{corollary}

For $g(\zeta)\in L^2(\Gamma,|\mathrm{d}\zeta|)$, we define the Cauchy integral, 
or Cauchy transform, of $g$ on $\Gamma$ as
\[Cg(w)= G(w)
  =\frac1{2\pi\mathrm{i}} \int_\Gamma 
     \frac{g(\zeta)\,\mathrm{d}\zeta}{\zeta-w},
  \quad \text{for } w\in\Omega_\pm,\]
then $G(w)$ is holomorphic on $\Omega_\pm$, and
\[G'(w)=\frac1{2\pi\mathrm{i}} \int_\Gamma 
      \frac{g(\zeta)\,\mathrm{d}\zeta}{(\zeta-w)^2}.\]
The two estimates in Theorem~\ref{thm-170727-1750} and 
Theorem~\ref{thm-170727-2120} could now be combined to yield a proof of 
the following theorem, which shows the $L^2$ boundedness of Cauchy integrals 
on Lipschitz curves. The proof also implies that $G\in\mathcal{T}(\Omega_\pm)$.

\begin{theorem}\label{thm-170728-0830}
  If $g(\zeta)\in L^2(\Gamma,|\mathrm{d}\zeta|)$ and $G(w)$ the Cauchy integral 
  of $g$ on $\Gamma$, then 
  \[\sup_{\tau>0} \lVert G(\cdot\pm\mathrm{i}\tau)
        \rVert_{L^2(\Gamma,|\mathrm{d}\zeta|)}
    \leqslant C(1+M^2)\lVert g\rVert_{L^2(\Gamma,|\mathrm{d}\zeta|)},\]
  where $C=196\mathrm{e}^{4\theta_0}$.
\end{theorem}

\begin{proof}
  We first assume $g$ is compactly supported on $\Gamma$, and suppose that 
  $E= \mathrm{supp}\,g\subset D(0,R)$, where $R>0$. If $|w|>2R$, then 
  $|w-\zeta|> \frac12|w|$ for $\zeta\in E$, and
  \begin{align*}
    |G(w)|
    &\leqslant \frac1{2\pi} \int_E 
         \frac{|g(\zeta)|}{|\zeta-w|}|\mathrm{d}\zeta|                  \\
    &\leqslant \frac1{\pi|w|} \int_E |g(\zeta)|\,|\mathrm{d}\zeta|      \\
    &\leqslant \frac1{\pi|w|} \Big(\int_E 
         |g(\zeta)|^2\,|\mathrm{d}\zeta|\Big)^{\frac12}
         \Big(\int_E |\mathrm{d}\zeta|\Big)^{\frac12}                    \\
    &\leqslant \frac1{\pi|w|} \lVert g\rVert_{L^2(\Gamma,|\mathrm{d}\zeta|)}
         (2R\sqrt{1+M^2})^{\frac12}                                      \\
    &= \frac{A}{|w|},
  \end{align*}
  where $A= \frac1{\pi} \lVert g\rVert_{L^2(\Gamma,|\mathrm{d}\zeta|)}
    (2R\sqrt{1+M^2})^{\frac12}$.
  We also have 
  \begin{align*}
    |G'(w)|
    &\leqslant \frac1{2\pi} \int_E 
         \frac{|g(\zeta)|}{|\zeta-w|^2}|\mathrm{d}\zeta|                \\
    &\leqslant \frac2{\pi|w|^2} \int_E |g(\zeta)|\,|\mathrm{d}\zeta|    \\
    &= \frac{2A}{|w|^2},
  \end{align*}  
  then $G\in\mathcal{T}(\Omega_\pm)$.
  
  Next we will focus on the case of $G\in\mathcal{T}(\Omega_+)$, and let
  \[B= \{f\in L^2(\Omega_+,\mathrm{d}\nu)\colon 
      \lVert f\rVert_{L^2(\Omega_+,\mathrm{d}\nu)}\leqslant 1,
      f \text{ is compactly supported in } \Omega_+\},\]
  then 
  \[\lVert G'\rVert_{L^2(\Omega_+,\mathrm{d}\nu)}
    = \sup_{f\in B} \Big|\iint_{\Omega_+} 
          G'\overline{f}\,\mathrm{d}\nu\Big|.\]
  Fix $\tau>0$, by Theorem~\ref{thm-170727-1750}, Fubini's theorem and 
  Theorem~\ref{thm-170727-2120}, we obtain,
  \begin{align*}
    &\lVert G(\cdot+\mathrm{i}\tau)\rVert_{L^2(\Gamma,|\mathrm{d}\zeta|)} \\
    \leqslant{}& C_1\sqrt{1+M^2} 
         \lVert G'\rVert_{L^2(\Omega_+,\mathrm{d}\nu)}                \\
    ={}& \frac{C_1\sqrt{1+M^2}}{2\pi} \sup_{f\in B} \Big|\iint_{\Omega_+} 
         \Big(\int_\Gamma \frac{g(\zeta)\,\mathrm{d}\zeta}{(\zeta-w_1)^2}\Big) 
         \overline{f(w_1)}d(w_1)\,\mathrm{d}\lambda(w_1)\Big|          \\
    ={}& \frac{C_1\sqrt{1+M^2}}{2\pi} \sup_{f\in B} \Big|
         \int_\Gamma g(\zeta)(T\overline{f})(\zeta)\,\mathrm{d}\zeta\Big| \\
    \leqslant{}& \frac{C_1\sqrt{1+M^2}}{2\pi} \sup_{f\in B} 
         \Big(\lVert g\rVert_{L^2(\Gamma,|\mathrm{d}\zeta|)}
         \lVert T\overline{f}\rVert_{L^2(\Gamma,|\mathrm{d}\zeta|)}\Big)   \\
    \leqslant{}& \frac{C_1C_2(1+M^2)}{2\pi} 
         \lVert g\rVert_{L^2(\Gamma,|\mathrm{d}\zeta|)}
         \sup_{f\in B} \lVert f\rVert_{L^2(\Omega_+,\mathrm{d}\nu)}    \\
    \leqslant{}& C(1+M^2)\lVert g\rVert_{L^2(\Gamma,|\mathrm{d}\zeta|)},
  \end{align*}
  where $C_1= 7\mathrm{e}^{2\theta_0}$, $C_2= 56\pi\mathrm{e}^{2\theta_0}$ and
  $C= \frac1{2\pi}C_1C_2= 196\mathrm{e}^{4\theta_0}$.
  
  In the general case, we let $g_n(\zeta)= \chi_{D(0,n)}g(\zeta)$ 
  for $\zeta\in\Gamma$, where $n>0$ and $\chi$ is 
  the characteristic function of a set, then $g_n$ is 
  compactly supported on $\Gamma$, and
  $\lVert g_n-g\rVert_{L^2(\Gamma,|\mathrm{d}\zeta|)}\to 0$ as $n\to\infty$. 
  For $\tau>0$ and $\zeta_0\in\Gamma$ both fixed, 
  let $w_0=\zeta_0+\mathrm{i}\tau\in\Omega_+$. Denote the Cauchy integral of 
  $g_n$ as $G_n$, then we have
  \begin{align*}
    |G_n(w_0)- G(w_0)|
    &\leqslant \frac1{2\pi} \int_\Gamma 
         \frac{|g_n(\zeta)-g(\zeta)|}{\zeta-w_0} |\mathrm{d}\zeta|      \\
    &\leqslant \frac1{2\pi} \lVert g_n-g\rVert_{L^2(\Gamma,|\mathrm{d}\zeta|)}
         \Big(\int_\Gamma \frac{|\mathrm{d}\zeta|}{|\zeta-w_0|^2}
             \Big)^{\frac12}                                          \\
    &\to 0 \text{ as } n\to\infty,
  \end{align*}
  thus, by Fatou's lemma, 
  \begin{align*}
    \lVert G(\cdot+\mathrm{i}\tau)\rVert_{L^2(\Gamma,|\mathrm{d}\zeta|)}
    &\leqslant \liminf_{n\to\infty} \lVert G_n(\cdot+\mathrm{i}\tau) 
          \rVert_{L^2(\Gamma,|\mathrm{d}\zeta|)}                     \\
    &\leqslant \liminf_{n\to\infty} C(1+M^2) \lVert g_n 
          \rVert_{L^2(\Gamma,|\mathrm{d}\zeta|)}                     \\
    &= C(1+M^2) \lVert g\rVert_{L^2(\Gamma,|\mathrm{d}\zeta|)},
  \end{align*}
  and the theorem follows.
\end{proof}

The following lemma is proved in~\cite{DL171}.

\begin{lemma}\label{lem-170622-2122}
  If $F(\zeta)\in L^p(\Gamma,|\mathrm{d}\zeta|)$, and $u_0$ is 
  the Lebesgue point of $F(u+\mathrm{i}a(u))$ such that 
  $\zeta'(u_0)= |\zeta'(u_0)|\mathrm{e}^{\mathrm{i}\phi_0}$ exists,
  where $\phi_0\in(-\theta_0,\theta_0)$, 
  then for any $\phi\in(0,\frac{\pi}2)$, we have
  \[\lim_{\substack{z+\zeta_0\in\Omega_\phi(\zeta_0)\cap\Omega^+,\\ 
        z\to 0}}
    \int_\Gamma K_z(\zeta,\zeta_0)F(\zeta)\,\mathrm{d}\zeta
    = F(\zeta_0).\]
\end{lemma}

Now we could proof that $L^2(\Gamma,|\mathrm{d}\zeta|)$ is the sum of 
$H^2(\Omega_\pm)$ in the non-tangential boundary limit sense.

\begin{corollary}
  Every function in $L^2(\Gamma,|\mathrm{d}\zeta|)$ is (a.e.\@ on $\Gamma$) 
  the sum of the non-tangential boundary limit of two functions in 
  $H^2(\Omega_+)$ and $H^2(\Omega_-)$, respectively, 
  or we could simply write
  \[L^2(\Gamma,|\mathrm{d}\zeta|)= H^2(\Omega_+)+ H^2(\Omega_-).\]
\end{corollary}

\begin{proof}
  For $g\in L^2(\Gamma,|\mathrm{d}\zeta|)$, let
  \[G_1(w_1)
    = \frac1{2\pi\mathrm{i}} \int_\Gamma 
        \frac{g(\zeta)\,\mathrm{d}\zeta}{\zeta-w_1},
    \quad \text{for } w_1\in\Omega_+,\]
  and
  \[G_2(w_2)
    = \frac1{2\pi\mathrm{i}} \int_\Gamma 
        \frac{g(\zeta)\,\mathrm{d}\zeta}{\zeta-w_2},
    \quad \text{for } w_2\in\Omega_-,\] 
  then both $G_1(w_1)$ and $G_2(w_2)$ are anaylitc~\cite{DL171}. 
  By Theorem~\ref{thm-170728-0830}, there exists constant~$C$, such that 
  \[\sup_{\tau>0} \lVert G_i(\cdot+\mathrm{i}\tau)
        \rVert_{L^2(\Gamma,|\mathrm{d}\zeta|)}
    \leqslant C(1+M^2)\lVert g\rVert_{L^2(\Gamma,|\mathrm{d}\zeta|)},\quad
    \text{for } i=1, 2.\] 
  It means that $G_1\in H^2(\Omega_+)$ 
  and $G_2\in H^2(\Omega_-)$, thus both of them have non-tangential 
  boundary limit a.e.\@ on $\Gamma$. We still denote the limit functions as 
  $G_1$ and $G_2$, respectively.
  
  Now suppose $u_0$ is the Lebesgue point of $g(\zeta(u))$ and $\zeta'(u_0)$ 
  exists. Let $\zeta_0= \zeta(u_0)$, $w_1=\zeta_0+z$ and $w_2=\zeta_0-z$, 
  where $z\in\mathbb{C}$ and $|z|$ is sufficiently small such that 
  $w_1\in\Omega_+$ and $w_2\in\Omega_-$, then
  \begin{align*}
    G_1(w_1)- G_2(w_2)
    &= G_1(\zeta_0+z)- G_2(\zeta_0-z)                                    \\
    &= \frac1{2\pi\mathrm{i}} \int_\Gamma \bigg(
            \frac{g(\zeta)}{\zeta-(\zeta_0+z)}
            - \frac{g(\zeta)}{\zeta-(\zeta_0-z)}
            \bigg)\mathrm{d}\zeta                                       \\
    &= \int_\Gamma K_z(\zeta,\zeta_0)g(\zeta)\,\mathrm{d}\zeta.
  \end{align*}
  Lemma~\ref{lem-170622-2122} implies that
  \[\lim_{z\to 0} |g(\zeta_0)- (G_1(w_1)- G_2(w_2))|= 0,\]
  and $g(\zeta_0)= G_1(\zeta_0)- G_2(\zeta_0)$ follows.
  Thus the corollary is proved.
\end{proof}

\section{The special case of ``$M=0$''}

In this section, we will obtain a more accurate upper boundary of 
the norm of Cauchy tranform under the assumption that
$\lVert a'\rVert_\infty= M= 0$. Notice that in this case, we have
\[\mathrm{d}\mu(z)
  = |y|\mathrm{d}\lambda(z)
  = d(z)\mathrm{d}\lambda(z)
  = \mathrm{d}\nu(z),\]
$\Gamma= \mathbb{R}$ and $\Omega_\pm= \mathbb{C}_\pm$.

\begin{theorem}\label{thm-170730-0900}
  If $F\in\mathcal{T}(\mathbb{C}_+)$ and $\tau>0$ is fixed, then 
  \[\lVert F(\cdot+\mathrm{i}\tau)\rVert_{L^2(\mathbb{R},\mathrm{d}x)}
    \leqslant 2\lVert F'\rVert_{L^2(\mathbb{C}_+,\mathrm{d}\mu)}.\]
  If, furthermore, $F$ is continuous to the boundary, then we also have
  \[\lVert F\rVert_{L^2(\mathbb{R},\mathrm{d}x)}
    = 2\lVert F'\rVert_{L^2(\mathbb{C}_+,\mathrm{d}\mu)}.\]
\end{theorem}

\begin{proof}
  The continuous case is just Corollary~\ref{cor-170727-1600}, 
  since now $\Phi_+(z)=z$ and $\Delta(|F|^2)= 4|F'|^2$.
  The general case is proved in the same way 
  as in Theorem~\ref{thm-170727-1750}.
\end{proof}

\begin{corollary}\label{cor-170730-0910}
  If $F\in\mathcal{T}(\mathbb{C}_-)$ and $\tau>0$ is fixed, then 
  \[\lVert F(\cdot-\mathrm{i}\tau)\rVert_{L^2(\mathbb{R},\mathrm{d}x)}
    \leqslant 2\lVert F'\rVert_{L^2(\mathbb{C}_-,\mathrm{d}\mu)}.\]
  If, furthermore, $F$ is continuous to the boundary, then we also have
  \[\lVert F\rVert_{L^2(\mathbb{R},\mathrm{d}x)}
    = 2\lVert F'\rVert_{L^2(\mathbb{C}_-,\mathrm{d}\mu)}.\]
\end{corollary}

The ``$M=0$'' version of Theorem~\ref{thm-170727-2120} is 
the following theorem.

\begin{theorem}\label{thm-170730-0920}
  Let $f\in L^2(\mathbb{C}_+,\mathrm{d}\mu)$ be compactly supported, and define, 
  for $z_2\in\overline{\mathbb{C}_-}$,
  \[Tf(z_2)= \iint_{\mathbb{C}_+} \frac{f(z_1)d(z_1)}{(z_1-z_2)^2} 
      \,\mathrm{d}\lambda(z_1)
    = \iint_{\mathbb{C}_+} \frac{f(z_1)\,\mathrm{d}\mu(z_1)}{(z_1-z_2)^2}.\]
  Then, $\lVert Tf\rVert_{L^2(\Gamma,|\mathrm{d}\zeta|)}
    \leqslant 4\pi\lVert f\rVert_{L^2(\mathbb{C}_+,\mathrm{d}\mu)}$.
\end{theorem}

\begin{proof}
  We could still verify that $Tf\in\mathcal{T}(\mathbb{C}_-)$ and 
  is continuous to the boundary $\mathbb{R}$, then, 
  by Corollary~\ref{cor-170730-0910}, 
  \[\lVert Tf\rVert_{L^2(\mathbb{R},\mathrm{d}x)}
    = 2\lVert (Tf)'\rVert_{L^2(\mathbb{C}_-,\mathrm{d}\mu)}.\]
  
  Define an operator $S\colon L^2(\mathbb{C}_+,\mathrm{d}\lambda)\to 
    L^2(\mathbb{C}_-,\mathrm{d}\lambda)$ by
  \begin{align*}
    SF(z_2)
    &= d(z_2)^{\frac12} \iint_{\mathbb{C}_+} 
         \frac{F(z_1)d(z_1)^{\frac12}}{|z_1-z_2|^3}\,\mathrm{d}\lambda(z_1)\\
    &= \iint_{\mathbb{C}_+} K(z_1,z_2)F(z_1)\,\mathrm{d}\lambda(z_1),
  \end{align*}
  where $z_1\in\mathbb{C}_+$, $z_2\in\mathbb{C}_-$, 
  and $K(z_1,z_2)= d(z_1)^{\frac12} d(z_2)^{\frac12} |z_1-z_2|^{-3}$.
  Let $z_1=x_1+\mathrm{i}y_1$, $z_2=x_2+\mathrm{i}y_2$, we have
  $d(z_1)= y_1$ and $d(z_2)= |y_2|$ by the definition of $d(z)$. 
  Then fix $z_2$, 
  \begin{align*}
    \iint_{\mathbb{C}_+} K(z_1,z_2)\,\mathrm{d}\lambda(z_1)
    &= \iint_{\mathbb{C}_+} \frac{|y_2|^{\frac12}y_1^{\frac12}
             \,\mathrm{d}\lambda(z_1)}
          {\big((x_1-x_2)^2+(y_1-y_2)^2\big)^{\frac32}}              \\
    &= \int_0^{+\infty}\!\!\int_{-\infty}^{+\infty} \frac{\mathrm{d}x_1}
              {|(x_1-x_2)^2+(y_1-y_2)^2|^{\frac32}}
          |y_2|^{\frac12}y_1^{\frac12}\,\mathrm{d}y_1           \\
    &= \int_0^{+\infty} \frac{|y_2|^{\frac12}y_1^{\frac12}}{(y_1-y_2)^2}
          \,\mathrm{d}y_1 \int_{-\infty}^{+\infty} 
          \frac{\mathrm{d}t}{(t^2+1)^{\frac32}}                    \\
    &= \frac{|y_2|^{\frac12}\cdot|y_2|^{\frac32}}{|y_2|^2}
          \int_0^{+\infty} \frac{t^{\frac12}\,\mathrm{d}t}{(t+1)^2} 
          \cdot 2\int_0^{+\infty} \frac{\frac12 t^{-\frac12}\,\mathrm{d}t} 
              {(t+1)^{\frac32}}                                   \\
    &= \int_0^{+\infty} \frac{t^{\frac12}\,\mathrm{d}t}{(t+1)^2} 
       \cdot \int_0^{+\infty} \frac{t^{-\frac12}\,\mathrm{d}t} 
             {(t+1)^{\frac32}}.
  \end{align*}
  Let $s=\frac{t}{1+t}$ for $t\in (0,+\infty)$, then $s\in(0,1)$, 
  \[t= \frac{s}{1-s},\quad 
    t+1= \frac1{1-s},\quad
    \text{and }\mathrm{d}t= \frac{\mathrm{d}s}{(1-s)^2},\]
  By invoking the Euler's Gamma function $\Gamma(\cdot)$ 
  and Beta function $B(\cdot)$, we have
  \begin{align*}
    \int_0^{+\infty} \frac{t^{\frac12}\,\mathrm{d}t}{(t+1)^2} 
    &= \int_0^1 \frac{s^{\frac12}(1-s)^{-\frac12}}{(1-s)^{-2}}
         \cdot (1-s)^{-2}\,\mathrm{d}s                        \\
    &= \int_0^1 s^{\frac12}(1-s)^{-\frac12}\,\mathrm{d}s      \\
    &= B\Big(\frac32,\frac12\Big)
     = \frac{\Gamma(\frac32)\Gamma(\frac12)}{\Gamma(2)}       \\
    &= \frac12\Gamma\Big(\frac12\Big)^2
     = \frac\pi2,
  \end{align*}
  and 
  \begin{align*}
    \int_0^{+\infty} \frac{t^{-\frac12}\,\mathrm{d}t}{(t+1)^{\frac32}} 
    &= \int_0^1 \frac{s^{-\frac12}(1-s)^{\frac12}}{(1-s)^{-\frac32}}
         \cdot (1-s)^{-2}\,\mathrm{d}s                        \\
    &= \int_0^1 s^{-\frac12}\,\mathrm{d}s      \\
    &= 2,
  \end{align*}
  then
  \[\iint_{\mathbb{C}_+} K(z_1,z_2)\,\mathrm{d}\lambda(z_1)
    = \frac\pi2\cdot 2
    = \pi.\]
  The same computation yields that, for fixed $z_1\in\mathbb{C}_+$,
  \[\iint_{\mathbb{C}_-} K(z_1,z_2)\,\mathrm{d}\lambda(z_2)=\pi.\]
  By Schur's lemma, $S$ is a bounded operator from 
  $L^2(\mathbb{C}_+,\mathrm{d}\lambda)$ to $L^2(\mathbb{C}_-,\mathrm{d}\lambda)$, 
  and $\lVert S\rVert\leqslant \pi$.
  
  Let $F(z_1)= f(z_1)d(z_1)^{\frac12}$, then 
  \[\lVert SF\rVert_{L^2(\mathbb{C}_-,\mathrm{d}\lambda)}
    \leqslant \pi\lVert F\rVert_{L^2(\mathbb{C}_+,\mathrm{d}\lambda)}
    = \pi\lVert f\rVert_{L^2(\mathbb{C}_+,\mathrm{d}\mu)},\]
  and 
  \[|(Tf)'(z_2)|
    \leqslant 2\iint_{\mathbb{C}_+} \frac{|f(z_1)|d(z_1)}{|z_1-z_2|^3} 
            \,\mathrm{d}\lambda(z_1) 
    = \frac{2(SF)(z_2)}{d(z_2)^{\frac12}},\]
  which follows that,
  \begin{align*}
    \lVert (Tf)'\rVert_{L^2(\mathbb{C}_-,\mathrm{d}\mu)}
    &\leqslant 2\lVert d^{-\frac12}SF\rVert_{L^2(\mathbb{C}_-,\mathrm{d}\mu)} \\
    &= 2\lVert SF\rVert_{L^2(\mathbb{C}_-,\mathrm{d}\lambda)}                 \\
    &\leqslant 2\pi\lVert f\rVert_{L^2(\mathbb{C}_+,\mathrm{d}\mu)},
  \end{align*}
  Remeber that $\lVert Tf\rVert_{L^2(\mathbb{R},\mathrm{d}x)}
    = 2\lVert (Tf)'\rVert_{L^2(\mathbb{C}_-,\mathrm{d}\mu)}$, then
  \[\lVert Tf\rVert_{L^2(\mathbb{R},\mathrm{d}x)}
    \leqslant 4\pi\lVert f\rVert_{L^2(\mathbb{C}_+,\mathrm{d}\mu)},\]
  and the theorem is proved.
\end{proof}

\begin{corollary}
  Let $f\in L^2(\mathbb{C}_-,\mathrm{d}\mu)$ be compactly supported, and define, 
  for $z_1\in\overline{\mathbb{C}_+}$,
  \[Tf(z_1)= \iint_{\mathbb{C}_+} \frac{f(z_2)d(z_2)}{(z_2-z_1)^2} 
      \,\mathrm{d}\lambda(z_2)
    = \iint_{\mathbb{C}_+} \frac{f(z_2)\,\mathrm{d}\mu(z_2)}{(z_2-z_1)^2}.\]
  Then, $\lVert Tf\rVert_{L^2(\Gamma,|\mathrm{d}\zeta|)}
    \leqslant 4\pi\lVert f\rVert_{L^2(\mathbb{C}_+,\mathrm{d}\mu)}$.
\end{corollary}

Now we could proof the boundedness of Cauchy integral on $\mathbb{R}$, 
which is a special case of Theorem~\ref{thm-170728-0830}.
\begin{theorem}
  If $g(x)\in L^2(\mathbb{R},\mathrm{d}x)$ and $G(z)$ the Cauchy integral 
  of $g$ on $\mathbb{R}$, that is 
  \[G(z)
    =\frac1{2\pi\mathrm{i}} \int_\Gamma 
       \frac{g(t)\,\mathrm{d}t}{t-z},
    \quad \text{for } z\in\mathbb{C}_\pm\]
  then 
  \[\sup_{\tau>0} \lVert G(\cdot\pm\mathrm{i}\tau)
        \rVert_{L^2(\mathbb{R},\mathrm{d}x)}
    \leqslant 4\lVert g\rVert_{L^2(\mathbb{R},\mathrm{d}x)}.\]
\end{theorem}

\begin{proof}
  We will prove the theorem while supposing that $g$ is compactly supported 
  on $\mathbb{R}$, and omit the proof of the general case, which could be 
  treated by the same method as in Theorem~\ref{thm-170728-0830}. 
  It has been proved in that theorem that $G\in\mathcal{T}(\mathbb{C}_\pm)$ 
  if $g$ is non-zero on a compact interval of $\mathbb{R}$. We suppose that 
  $G\in\mathcal{T}(\mathbb{C}_+)$, and let
  \[B= \{f\in L^2(\mathbb{C}_+,\mathrm{d}\mu)\colon 
      \lVert f\rVert_{L^2(\mathbb{C}_+,\mathrm{d}\mu)}\leqslant 1,
      f \text{ is compactly supported in } \mathbb{C}_+\},\]
  then 
  \[\lVert G'\rVert_{L^2(\mathbb{C}_+,\mathrm{d}\mu)}
    = \sup_{f\in B} \Big|\iint_{\mathbb{C}_+} 
          G'\overline{f}\,\mathrm{d}\mu\Big|.\]
  Fix $\tau>0$, by Theorem~\ref{thm-170730-0900}, Fubini's theorem and 
  Theorem~\ref{thm-170730-0920}, we have,
  \begin{align*}
    &\lVert G(\cdot+\mathrm{i}\tau)\rVert_{L^2(\mathbb{R},\mathrm{d}x)} \\
    ={}& 2\lVert G'\rVert_{L^2(\mathbb{C}_+,\mathrm{d}\mu)}             \\
    ={}& \frac1\pi \sup_{f\in B} \Big|\iint_{\mathbb{C}_+} 
         \Big(\int_\mathbb{R} \frac{g(t)\,\mathrm{d}t}{(t-z_1)^2}\Big) 
         \overline{f(z_1)}d(z_1)\,\mathrm{d}\lambda(z_1)\Big|          \\
    ={}& \frac1\pi \sup_{f\in B} \Big|
         \int_\mathbb{R} g(t)(T\overline{f})(t)\,\mathrm{d}t\Big| \\
    \leqslant{}& \frac1\pi \sup_{f\in B} 
         \Big(\lVert g\rVert_{L^2(\mathbb{R},\mathrm{d}x)}
         \lVert T\overline{f}\rVert_{L^2(\mathbb{R},\mathrm{d}x)}\Big)   \\
    \leqslant{}& \frac{4\pi}{\pi} 
         \lVert g\rVert_{L^2(\mathbb{R},\mathrm{d}x)}
         \sup_{f\in B} \lVert f\rVert_{L^2(\mathbb{C}_+,\mathrm{d}\mu)}    \\
    \leqslant{}& 4\lVert g\rVert_{L^2(\mathbb{R},\mathrm{d}x)},
  \end{align*}
  and this proves the theorem.
\end{proof}

\section*{Funding}
This work is supported by National Natural Science Foundation 
of China(Grant No.\@ 11271045).

\end{document}